\documentclass[12pt, english, a4paper]{amsart} % Proof style

\pdfoutput=0

\usepackage{graphicx}
\usepackage{babel}
\usepackage{amsmath}
\usepackage{amsthm}
\usepackage{amssymb}
\usepackage{amsfonts}
\usepackage{latexsym}

\newtheorem{newstatement}{newstatement}
\newtheorem{definition}[newstatement]{Definition}
\newtheorem{lemma}[newstatement]{Lemma}
\newtheorem{theorem}[newstatement]{Theorem}
\newtheorem{corollary}[newstatement]{Corollary}

\newtheorem{remark}[newstatement]{Remark}
\newtheorem{proposition}[newstatement]{Proposition}

\def\fieldstyle{\Bbb} %%% Modificare per definire \R, \C

\def\p_#1{p_{\kern-1.5pt_#1}}
\def\emph#1{{\em #1}}
\def\textbf#1{{\bf #1}}
\def\MR#1{\relax}

\newcommand{\Cl}{\mathop{\mathrm{Cl}}\nolimits} 
\newcommand{\Int}{\mathop{\mathrm{Int}}\nolimits} 
\newcommand{\Bd}{\partial}

\newcommand{\Map}{\mathop{\cal M}\nolimits}

\newcommand{\M}{\mathop{\cal M}\nolimits}

\newcommand{\sign}{\mathop{\mathrm{sign}}\nolimits}

\newcommand{\R}{\fieldstyle R}
\newcommand{\C}{\fieldstyle C}

\newcommand{\CP}{\C\mathrm P}

\def\(#1){({\em #1\/})}

\def\varemptyset{{\text{\raise.21ex\hbox{$\not$}}\mkern.15mu\mathrm{O}\mkern.15mu}}

\let\emptyset\varemptyset
\let\geq\geqslant
\let\leq\leqslant

\let\phi\varphi
\let\epsilon\varepsilon

\makeatletter

\renewcommand{\paragraph}{\@startsection%
{paragraph}% name
{4}% level
{0mm}% indent
{\bigskipamount}% beforeskip
{-1.25ex}% afterskip
{\normalsize\bf}}% style

\def\provedboxcontents#1{$\square$}

\makeatother

\usepackage[left=2.5cm, top=3.2cm, bottom=3.2cm, right=2.5cm]{geometry}

\makeatletter

\def\ft{\@ifnextchar[{\ft@s}{\ft@}}
\def\ft@{\ft@@@s[\f@size]}
\def\ft@s[{\@ifnextchar{a}{\ft@sz[}{\ft@@s[}}
\def\ft@@s[{\@ifnextchar{s}{\ft@sz[}{\ft@@@s[}}
\def\ft@@@s[#1]{\ft@sz[at #1pt]}
\def\ft@sz[#1]#2{\font\fonttemp=#2 #1\fonttemp\ignorespaces}

\makeatother

\begin{document}
\def\calstyle#1{{\text{\ft{eusm10}#1}}}

\let\cal\calstyle

\title[]{\bf\large UNIVERSAL LEFSCHETZ FIBRATIONS\\[6pt] OVER BOUNDED 
SURFACES}

\author{Daniele Zuddas}

\address{Universit\`a di Cagliari\\
Dipartimento di Matematica e Informatica\\
Via Ospedale 72\\ 09124 Cagliari (Italy)}

\email{d.zuddas@gmail.com}

\thanks{I would like to thanks Andrea Loi for helpful conversations and 
suggestions about the manuscript. Thanks to Regione Autonoma della Sardegna for support with funds from PO Sardegna FSE~2007--2013 and L.R.~7/2007 ``Promotion of scientific research and technological innovation in Sardinia''. Also thanks to ESF for short visit grants within the program ``Contact and Symplectic Topology".}

\date{}


\begin{abstract}
	In analogy with the vector bundle theory  we define universal  and strongly universal Lefschetz fibrations over bounded surfaces. After giving a characterization of these fibrations we construct very special strongly universal Lefschetz fibrations when the fiber is the torus or an orientable surface with connected boundary and the base surface is the disk. As a by-product we also get some immersion results for $4$-dimensional $2$-handlebodies.
	
\end{abstract}
\maketitle

\section*{Introduction}

Consider a smooth 4-manifold $V$ and a surface $S$. Let $f : V \to S$ be a (possibly achiral) smooth Lefschetz fibration with singular values set $A_f\subset S$ and regular fiber $F \cong F_{g,b}$, the compact connected orientable surface of genus $g$ with $b$ boundary components. Let $G$ be another surface. We assume that $V$, $S$ and $G$ are compact, connected and oriented with (possibly empty) boundary.

\begin{definition}
	We say that a smooth map $q : G \to S$ with regular values set $R_q \subset S$ is $f$-regular if $q(\Bd G) \cap A_f = \emptyset$ and $A_f \subset R_q$.
\end{definition}

In other words, $q$ is $f$-regular if and only if $q$ and $q_{|\Bd G}$ are transverse to $f$.

If $q$ is $f$-regular then $\widetilde V = \{(g, v) \in G \times V \; |\; q(g) = f(v)\}$ is a 4-manifold and the map $\widetilde f :\widetilde V \to G$ given by $\widetilde f(g, v) = g$ is a Lefschetz fibration. Moreover $\widetilde q : \widetilde V \to V$, $\widetilde q(g, v) = v$, is a fiber preserving map which sends each fiber of $\widetilde f$ diffeomorphically onto a fiber of $f$, so the regular fiber of $\widetilde f$ can be identified with $F$. We get the following commutative diagram

\medskip\centerline{\includegraphics{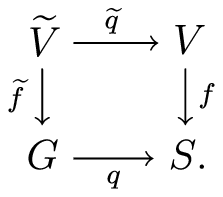}}\medskip

We say that $\widetilde f : \widetilde V \to G$ is the {\em pullback} of $f$ by $q$ and we make use of the notation $\widetilde f = q^*(f)$.

Two Lefschetz fibrations $f_1 : V_1 \to S_1$ and $f_2 : V_2 \to S_2$ are said 
{\em equivalent} if there are orientation-preserving diffeomorphisms $\phi : S_1 \to S_2$ and $\widetilde \phi : V_1 \to V_2$ such that $\phi \circ f_1 = f_2 \circ \widetilde\phi$. The equivalence class of $f$ will be indicated by $[f]$.

We say that a Lefschetz fibration $f$ is {\em allowable} if the vanishing cycles of $f$ with respect to a (and hence to any) Hurwitz system for $A_f$ are homologically essential in $F$. We consider only allowable Lefschetz fibrations, if not differently stated.

Given $f$ we define the set $L(f) = \{[q^*(f)]\}$ where $q$ runs over the $f$-regular maps $q : G \to S$ and $G$ runs over the compact, connected, oriented surfaces. 
Analogously, we define the set $SL(f) = \{[q^*(f)]\} \subset L(f)$ where $q$ runs over the $f$-regular orientation-preserving immersions $q : G \to S$, with $G$ as above.

\begin{definition}\label{univ/def}
	A Lefschetz fibration $u : U \to S$ with regular fiber $F$ is said universal (resp. strongly universal) if every  Lefschetz fibration over a surface with non-empty boundary and with regular fiber diffeomorphic to $F$ belongs to a class of $L(u)$ (resp. $SL(u))$.
\end{definition}

In other words $u$ is universal (resp. strongly universal) if and only if any Lefschetz fibration over a bounded surface and with the same fiber can be obtained as the pullback of $u$ by a $u$-regular map (resp. immersion).
Note that this notion of universality is analogous to that in the theory of vector bundles \cite{MS74}.

We denote by $\M_{g,b}$ the mapping class group of $F_{g,b}$ whose elements are the isotopy classes of orientation-preserving diffeomorphisms of $F_{g,b}$ which keep the boundary fixed pointwise (assuming isotopy through such diffeomorphisms).

It is well-known that for a Lefschetz fibration $f : V \to S$ with regular fiber $F_{g,b}$ the monodromy of a meridian\footnote{A meridian of a finite subset $A \subset \Int S$ is an element of $\pi_1(S - A)$ which can be represented by the oriented boundary of an embedded disk in $S$ which intersects $A$ in a single interior point, cf. Definition~\ref{immers-merid/def}.} of $A_f$ is a Dehn twist $\gamma \in \M_{g,b}$.

If $S$ is not simply connected, the monodromy of an element of $\pi_1(S - A_f)$ which is not a product of meridians is not necessarily the identity on $\Bd F_{g,b}$, and so it induces a permutation on the set of boundary components of $F_{g,b}$. We will denote by $\varSigma_b$ the permutation group of this set.

These considerations allow us to define three kind of monodromies.
Let $H_f \lhd \pi_1(S - A_f)$ be the smallest normal subgroup of $\pi_1(S - A_f)$ which contains all the meridians of $A_f$.
The {\em Lefschetz monodromy} of $f$ is the group homomorphism $\omega_f : H_f  \to \M_{g,b}$ which sends a meridian of $A_f$ to the corresponding Dehn twist in the standard way \cite{GS99}.

Let $\widehat \M_{g,b}$ be the extended mapping class group of $F_{g,b}$, namely the group of all isotopy classes of orientation-preserving self-diffeomorphisms of $F_{g,b}$.
The {\sl bundle monodromy} $\widehat\omega_f : \pi_1(S - A_f) \to \widehat \M_{g,b}$ is the monodromy of the locally trivial bundle $f_| : V - f^{-1}(A_f) \to S - A_f$.

We consider also the natural homomorphism $\sigma : \widehat\M_{g,b} \to 
\varSigma_b$ which sends an isotopy class to the permutation induced on the set of boundary components. The composition $\sigma \circ\widehat \omega_f$ passes to the quotient $\pi_1(S - A_f)/H_f \cong \pi_1(S)$ and gives a homomorphism $\omega_f^\sigma : \pi_1(S) \to \varSigma_b$ which we call the {\sl permutation monodromy} of $f$.

Let $\cal C_{g,b}$ be the set of equivalence classes of homologically 
essential simple closed curves in $\Int F_{g,b}$ up to orientation-preserving homeomorphisms of $F_{g,b}$.
It is well-known that $\cal C_{g,b}$ is finite. Moreover $\#\,\cal C_{g,b} = 1$ for $g \geq 1$ and $b \in \{0, 1\}$ \cite[Chapter~12]{L97}.

Now we state the main results of the paper.
In the following proposition  we  characterize the universal and strongly universal Lefschetz fibrations.

\begin{proposition}\label{main/thm}
	A Lefschetz fibration $u : U \to S$ with regular fiber $F_{g,b}$ is universal (resp. strongly universal)  if and only if the following two conditions $(1)$ and $(2)$ (resp. $(1)$ and $(2'))$ are satisfied:
	\begin{itemize}
		\item [$(1)$] $\omega_u$ and $\omega_u^\sigma$ are surjective;
		\item [$(2)$] any class of $\cal C_{g,b}$ can be represented by a vanishing cycle of $u$;
		\item [$(2')$]  any class of $\cal C_{g,b}$ contains at least two vanishing cycles of $u$ which correspond to singular points of opposite signs.
	\end{itemize}
	
	In particular, if $b \in \{0, 1\}$ and $\omega_u$ is surjective then $u$ is universal. If in addition $u$ admits a pair of opposite singular points, then $u$ is strongly universal.
\end{proposition}

As a remarkable simple consequence we have that (strongly) universal Lefschetz fibrations actually exist for any regular fiber $F_{g,b}$. Moreover, the surjectivity of $\omega^\sigma_u$ implies $b_1(S) \geq ($the minimum number of generators of $\varSigma_b)$, and this inequality is sharp. So we can assume that the base surface $S$ of a universal Lefschetz fibration is the disk for $b \leq 1$, the annulus for $b = 2$, and such that $b_1(S) = 2$ for $b \geq 3$.

Consider a knot $K \subset S^3$ and let $M(K, n)$ be the oriented 4-manifold obtained from $B^4$ by the addition of a 2-handle along $K$ with framing $n$.

In the following theorem we construct very special strongly universal Lefschetz fibrations when the fiber is the torus or $F_{g, 1}$ with $g\geq 1$. 

\begin{theorem}\label{strong-univ/thm}
	There is a strongly universal Lefschetz fibration $u_{g,b} : U_{g,b} \to B^2$ with fiber $F_{g,b}$ and with
	
	\smallskip
	\indent $U_{1,1} \cong  B^4$,\\
	\indent $U_{g, 1}  \cong M(O, 1) \text{ for } g \geq 2$, and\\
	\indent $U_{1,0} \cong M(E, 0)$,
	\smallskip
	
	\noindent where $O$ and $E$ denote respectively the trivial  and the figure eight knots in $S^3$.
\end{theorem}

\begin{corollary}\label{paral/cor}
	Let $f : V \to B^2$ be a Lefschetz fibration with fiber of genus one. Suppose that no vanishing cycle of $f$ disconnects the regular fiber with respect to some (and hence to any) Hurwitz system. Then $V$ immerses in $\R^4$ and so is parallelizable.
\end{corollary}

By means of Theorem~\ref{strong-univ/thm} we are able to give a new elementary proof of the following corollary. This was known since the work of Phillips \cite{P67} about submersions of open manifolds because there exists a bundle monomorphism $T\, V \to T\, \CP^2$ for any oriented 4-manifold $V$ which is homotopy equivalent to a CW-complex of dimension two (obtained by means of the classifying map to the complex universal vector bundle \cite{MS74}). 

\begin{corollary}\label{immers/cor}
	Any compact oriented 4-dimensional 2-handlebody\,\footnote{Recall that an $n$-dimensional $k$-handlebody is a smooth $n$-manifold built up with handles of index $\leq k$.} admits an o\-ri\-en\-ta\-tion-preserving immersion in $\CP^2$.
\end{corollary}

Universal maps in the context of Lefschetz fibrations over closed surfaces can be constructed in a different way. This generalization will be done in a forthcoming paper.

The paper consists of three other sections. In the next one we review some basic material on Lefschetz fibrations needed in the paper. Section~\ref{proof/sec} is dedicated to the proofs of our results, and in Section~\ref{remarks/sec} we give some final remarks and a comment on positive Lefschetz fibrations on Stein compact domains of dimension four. 

Throughout the paper we assume manifolds (with boundary) to be compact, oriented and connected if not differently stated. We will work in the $C^\infty$ category.

\section{Preliminaries}\label{preliminaries/sec}

Let $V$ be a 4-manifold (possibly with boundary and corners) and let $S$ be a surface. 

\begin{definition}\label{lf/def}
	A Lefschetz fibration $f : V \to S$ is a smooth map which satisfies the following three conditions: 
	\begin{enumerate}
		\item the singular set $\widetilde A_f\subset \Int V$ is finite and is mapped injectively onto the singular values set $A_f = f(\widetilde A_f) \subset \Int S$;
		
		\item the restriction $f_| : V - f^{-1}(A_f) \to S - A_f$ is a locally trivial bundle  with fiber a surface $F$ (the regular fiber of $f$);
		
		\item around each singular point $\widetilde a \in \widetilde A_f$, $f$ can be 
		locally expressed as the complex map $f(z_1, z_2) = z_1^2 + z_2^2$ for suitable chosen smooth local complex coordinates.
	\end{enumerate}
\end{definition}

If such local coordinates are orientation-preserving (resp. reversing), then $\widetilde a$ is said a {\em positive (resp. negative) singular point}, and $a = f(\widetilde a) \in A_f$ is said a {\em positive (resp. negative) singular value}. Observe that this positivity (resp. negativity) notion does not depend on the orientation of $S$. Obviously, around a negative singular point there are orientation-preserving local complex coordinates such that $f(z_1, z_2) = z_1^2 + \bar z_2^2$.

Most authors add the adjective `achiral' in presence of negative singular points. We prefer to simplify the terminology and so we do not follow this convention.

The orientations of $V$ and of $S$ induce an orientation on $F$ such that the locally trivial bundle associated to $f$ is oriented. We will always consider $F$ with this canonical orientation.

Let $F_{g,b}$ be the orientable surface of genus $g \geq 0$ with $b \geq 0$ boundary components. A Lefschetz fibration $f : V \to S$  with regular fiber $F = F_{g,b}$ is characterized by the {\em Lefschetz monodromy homomorphism} $\omega_f : H_f \to \Map_{g,b}$, which sends meridians of $A_f$ to Dehn twists, and by the {\em bundle monodromy homomorphism} $\widehat\omega_f : \pi_1(S - A_f) \to \widehat\M_{g,b}$ which is the monodromy of the bundle associated to $f$. Sometimes we use the term `monodromy' by leaving the precise meaning of it to the context.

It is well-known that the monodromy of a counterclockwise meridian of a positive (resp. negative) singular value is a right-handed (resp. left-handed) Dehn twist around a curve in $F$ \cite{GS99}. Such a curve is said to be a {\em vanishing cycle} for $f$.
We recall the following definition.

\begin{definition}\label{hurwitz/def} 
	A Hurwitz system for a cardinality $n$ subset $A \subset \Int S$ is a sequence $(\xi_1, \dots, \xi_n)$ of meridians of $A$ which freely generate $\pi_1(D - A)$ and such that the product $\xi_1 \cdots \xi_n$ is the homotopy class of the oriented boundary of $D$, where $D \subset S$ is a disk such that $A \subset \Int D$ and $*\in \Bd D$.
\end{definition}

If a Hurwitz system $(\xi_1, \dots, \xi_n)$ for $A_f \subset S$ is given, the set $A_f = \{a_1,\dots, a_n \}$ can be numbered accordingly so that $\xi_i$ is a meridian of $a_i$. It is determined a sequence of signed vanishing cycles $(c_1^\pm, \dots, c_n^\pm)$ (the {\em monodromy sequence} of $f$), where $c_i \subset F_{g,b}$ corresponds to the Dehn twist $\omega_f(\xi_i) \in \M_{g,b}$ and the sign of $c_i$ equals that of $a_i$ as a singular value of $f$. 
Clearly, $c_i$ is defined up to isotopy for all $i$. Sometimes the plus signs are understood.

The fact that the $c_i$'s are all homologically (or homotopically) essential in $F$ does not depend on the actual Hurwitz system, and so this is a property of the Lefschetz fibration $f$.

The monodromy sequence of $f : V \to B^2$ determines a handlebody 
decomposition of the total space as $V = (B^2 \times F) \cup H_1^2 \cup \cdots \cup H_n^2$ where $B^2 \times F$ is a trivialization of the bundle associated to $f$ over a subdisk contained in $B^2 - A_f$ and the 2-handle $H_i^2$ is attached to $B^2 \times F$ along the vanishing cycle $\{*_i\} \times c_i \subset S^1 \times F \subset 
\Bd (B^2 \times F)$ for a suitable subset $\{*_1, \dots, *_n\} \subset S^1$ cyclically ordered in the counterclockwise direction \cite{GS99}. The framing of $H^2_i$ with respect to the fiber $\{*_i\} \times F \subset \Bd(B^2 \times F)$ is $-\epsilon_i$ where $\epsilon_i = \pm 1$ is the sign of the singular point $a_i$. 

Note that $B^2 \times F$ can be decomposed as the union of a 0-handle, some 1-handles, and also a 2-handle in case $\Bd F = \emptyset$ starting from a handlebody decomposition of $F$ and making the product with the 2-dimensional 0-handle $B^2$.

In this paper we consider only the so called {\em relatively minimal} Lefschetz fibrations, namely those without homotopically trivial vanishing cycles. Then in our situation the monodromy sequence can be expressed also by a sequence of Dehn twists $(\gamma_1^{\epsilon_1}, \dots, \gamma_n^{\epsilon_n})$, where $\gamma_i = (\omega_f(\xi_i))^{\epsilon_i}$ is assumed to be right-handed.

Let $\mu : \M_{g,b} \to \widehat\M_{g,b}$ be the homomorphism such that $\mu([\phi]) = [\phi]$ for all $[\phi] \in \M_{g,b}$. We have the exact sequence $\M_{g,b} \stackrel\mu\to \widehat\M_{g,b} \stackrel\sigma\to \varSigma_b \to 0$ where $\sigma$ is the boundary permutation homomorphism defined in the Introduction. The monodromy homomorphisms $\omega_f$ and $\widehat\omega_f$ of 
a Lefschetz fibration $f$ satisfy $\widehat\omega_{f| H_f} = \mu \circ \omega_f$.

For a finite subset $A \subset \Int S$ we indicate by $H(S, A) \lhd \pi_1(S - A)$ the normal subgroup generated by the meridians of $A$. Given $S$, $A$ and two homomorphisms $\omega : H(S, A) \to \M_{g,b}$ and $\widehat\omega : \pi_1(S - A) \to \widehat\Map_{g,b}$ such that $\omega$ sends meridians to Dehn twists and $\mu\circ\omega = \widehat\omega_{|H(S,A)}$, there exists a Lefschetz fibration $f : V \to S$ with regular fiber $F_{g,b}$ such that $A_f = A$, $\omega_f = \omega$ and $\widehat\omega_f = \widehat\omega$. Moreover, 
such $f$ is unique up to equivalence by our relative minimality assumption, unless $S$ is closed and the fiber is a sphere or a torus (because in such cases the diffeomorphisms group of the fiber is not simply connected \cite{Gr73}).
In particular, if $S$ has boundary the 4-manifold $V$ is determined up to orientation-preserving diffeomorphisms. In \cite{APZ2011} we give a very explicit construction of $f$ starting from the monodromy sequence.

If $q : G \to S$ is $f$-regular with respect to a Lefschetz fibration $f : V \to S$ then the pullback $\widetilde f = q^*(f)$ satisfies $A_{\widetilde f} = q^{-1}(A_f)$, $\omega_{\widetilde f} = \omega_f \circ q_*$ and $\widehat\omega_{\widetilde f} = \widehat\omega_f \circ q_*$, where $q_* : \pi_1(G- A_{\widetilde f}) \to \pi_1(S - A_f)$ is the homomorphism induced by the restriction $q_| : G - A_{\widetilde f} \to S - A_f$. The base points are 
understood and are chosen so that $q(*') = *$ with $*'$ in the domain and $*$ in the codomain.

\begin{remark}
	The $f$-regularity of $q$ implies that $q_*(H_{\widetilde f}) \subset H_f$.
\end{remark}

Let $a \in A_f$ and $a' \in q^{-1}(a)$. Then $q$ is a local diffeomorphism around $a'$. It is immediate that the sign of $a'$ as a singular value of $\widetilde f$ is given by that of $a$ multiplied by the local degree of $q$ at $a'$, in other words $\sign(a') = \deg_{a'}(q)\cdot \sign(a)$.

In order to prove Theorem~\ref{strong-univ/thm} we recall the definition of {\em stabilizations} (the reader is referred to \cite{GS99} or \cite{APZ2011} for  details).

Given a Lefschetz fibration $f : V \to B^2$ whose regular fiber $F$ has non-empty boundary, we can construct a new Lefschetz fibration $f' : V' \to B^2$ by an operation called stabilization which is depicted in Figure~\ref{stabil/fig}. The new regular fiber $F' = F \cup H^1$ is obtained by attaching an orientable 1-handle $H^1$ to $F$, and the new monodromy sequence is given by the addition to the old one of a signed vanishing cycle $c^\pm$ which crosses $H^1$ geometrically once. 

The inverse operation is called {\em destabilization} and can be applied if there is a properly embedded arc $s$ in the regular fiber $F'$ of $f'$ which meets a single vanishing cycle $c$, and it does so geometrically once. The arc $s$ is the cocore of a 1-handle of $F'$. Let $F$ be $F'$ cut open along $s$, and let the new monodromy sequence be that of $f'$ with $c^\pm$ removed (no matter whichever is the sign).

In terms of handlebody decompositions, stabilizations (resp. destabilizations) correspond to the addition (resp. deletion) of a cancelling pair of 1 and 2-handles, hence $V \cong V'$.

\begin{figure}
	\centering\includegraphics{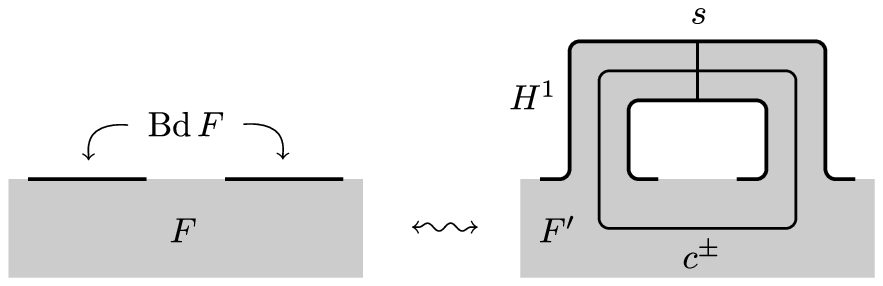}
	\caption{The (de)stabilization operation.}
	\label{stabil/fig}
\end{figure}

We end this section with the following straightforward propositions, needed in the proof of our main results.

\begin{proposition}
	Let $f : V \to S$ be a Lefschetz fibration and let $G$ be a surface. If $q_t : G \to S$, $t \in [0,1]$, is a homotopy through $f$-regular maps then $q_0^*(f) \cong q_1^*(f)$.
\end{proposition}

\begin{proposition}
	If $q : G \to S$ is an orientation-preserving immersion (resp. embedding) then the fibered map $\widetilde q : \widetilde V \to V$ associated to the pullback $q^*(f) : \widetilde V \to G$ is also an orientation-preserving immersion (resp. embedding).
\end{proposition}

\section{Proofs of main results}\label{proof/sec}

We first prove the following Lemma.

\begin{lemma}\label{van-cic/lem}
	Suppose that $u : U \to S$ satisfies conditions $(1)$ and $(2)$ (resp. $(1)$ and $(2'))$ of Proposition~\ref{main/thm}. Then each class of $\cal C_{g,b}$ can be represented by a vanishing cycle (resp. a vanishing cycle of prescribed sign) in any monodromy sequence of $u$.
\end{lemma}

\begin{proof}
	Let $(\xi_1, \dots, \xi_n)$ be a Hurwitz system for $A_u \subset S$ and let $\mathfrak c \in \cal C_{g,b}$. There is a meridian $\xi$ of $A_u$ such that $\omega_u(\xi)$ is a Dehn twist (resp. a Dehn twist of prescribed sign) around a curve $c \in \mathfrak c$. It is well-known that $\xi$ is conjugate to some $\xi_i$, hence $\xi = \tau \xi_i \tau^{-1}$ for some $\tau \in \pi_1(S - A_u)$. Let $c_i$ be the vanishing cycle of $\gamma_i$.
	
	Put $\phi = \widehat\omega_u(\tau)$, $\gamma = \widehat\omega_u(\xi)$ and $\gamma_i = \widehat\omega_u(\xi_i)$. We get $\gamma =\phi^{-1} \circ \gamma_i \circ \phi$ (because the standard right to left composition rule of maps differs from that in the fundamental group). There are two cases depending on whether $\gamma_i$ is or not the identity.
	
	If $\gamma_i$ is the identity then $\gamma$ is also the identity. It follows that $c$ and $c_i$ are boundary parallel (by the relatively minimal assumption they cannot be homotopically trivial), hence $c_i \in \mathfrak c$. If $\gamma_i$ is not the identity then $c_i =\phi(c) \in \mathfrak c$ \cite{W99}.
\end{proof}

We need also the following definition which gives a generalization of the notion of meridian.

\begin{definition}\label{immers-merid/def}
An immersed meridian for a finite subset $A \subset \Int S$ is an element of $\pi_1(S- A)$ which can be represented by the oriented boundary of an immersed disk $B \subset S$ such that $\#(B \cap A) = 1$.
\end{definition}

The immersed meridians of $A$ are precisely the conjugates in $\pi_1(S -A)$ of the meridians.

\begin{proof}[Proof of Proposition~\ref{main/thm}]
	We consider first the case of  universal Lefschetz fibrations.
	
	\noindent{\em `If' part.}
	Suppose that $u$ satisfies  conditions $(1)$ and $(2)$ of the statement. Consider a Lefschetz fibration $f : V \to G$  with regular fiber $F_{g,b}$, where $G$ is an oriented surface with non-empty boundary. We  are going to show  that $f \cong q^*(u)$ for some $u$-regular map $q : G \to S$. Without loss of generality we can assume $\Bd S \ne \emptyset$.
	
	There is a handlebody decomposition of $G$ with only one 0-handle $G^0$ and $l \geq 0$ 1-handles $G^1_i$, so $G = G^0 \cup G_1^1 \cup \cdots \cup G^1_l$. We can assume that $A_f \subset G^0$.
	
	Fix base points $* \in \Bd S$ and $*'\in \Bd G$ and let $\{\xi_1, \dots, \xi_n\}$ be a Hurwitz system for $A_u = \{a_1, \dots, a_n\} \subset S$. Fix also a set of free generators $\{\zeta_1,\dots, \zeta_k, \eta_1, \dots, \eta_l\}$ for $\pi_1(G - A_f)$ with $(\zeta_1,\dots, \zeta_k)$ a Hurwitz system for $A_f \subset G$. We assume that the $\zeta_i$'s are represented by meridians contained in $G^0$ and that $\eta_i$ is represented by an embedded loop (still denoted by $\eta_i$) which meets the 1-handle $G^1_i$ geometrically once and does not meet any other $G_j^1$ for $j \ne i$ as in Figure~\ref{surface/fig}. In this figure the 1-handles $G^1_i$'s are contained in the white lower box. We assume also that $\eta_i \cap \eta_j = \{*'\}$ for $i \ne j$.
	
	\begin{figure}
		\centering\includegraphics{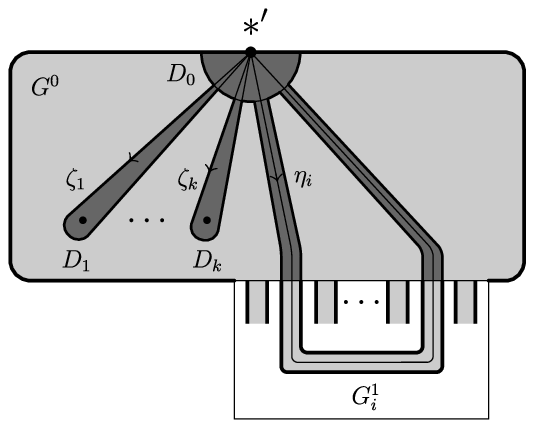}
		\caption{The generators $\zeta_i$'s and $\eta_i$'s.}
		\label{surface/fig}
	\end{figure}
	
	There are disks $D_1, \dots, D_k \subset G$ as those depicted in dark grey in the same Figure~\ref{surface/fig} such that $\Bd D_i \subset G - A_f$ represents $\zeta_i$ as a loop and $D_i \cap D_j = \{ *' \}$ for all $i \ne j$. There is also a disk $D_0 \subset G - A_f$ which is a neighborhood of $*'$ such that $D_0 \cap D_i \cong B^2$ and $D_0 \cap \eta_i \cong [0,1]$ for all $i$. Then $D = D_0 \cup D_1 \cup\dots\cup  D_k$ is diffeomorphic to $B^2$ (up to smoothing the corners).
	
	It follows that $G$ is diffeomorphic to the surface $G'$ obtained from $D$ by the addition of orientable 1-handles $G'_1, \dots, G'_l$ where $G'_i$ has attaching sphere the endpoints of the arc $\eta_i\cap D_0$ for any $i = 1, \dots, l$.
	
	Consider the Dehn twists $\gamma_i^{\epsilon_i} =\omega_u(\xi_i)$ and $\delta_i^{\sigma_i} = \omega_f(\zeta_i)$ around respectively the vanishing cycles $c_i$ (for $u$) and $d_i$ (for $f$), where $\epsilon_i, \sigma_i = \pm 1$ ($\gamma_i$ and $\delta_i$ are assumed to be right-handed).
	
	By Lemma~\ref{van-cic/lem} $d_i$ is equivalent in $\cal C_{g,b}$ to some $c_{j_i}$. It follows that $\delta_i = \lambda_i^{-1} \circ\gamma_{j_i} \circ \lambda_i$ for some $\lambda_i \in \widehat\M_{g,b}$ which sends $d_i$ to $c_{j_i}$ \cite{W99} (note that this composition is well-defined in $\M_{g,b}$ as the isotopy class of $\bar\lambda^{-1}_i \circ \bar\gamma_{j_i} \circ \bar\lambda_i$ where $\bar\lambda_i$ and $\bar\gamma_{j_i}$ are representatives of $\lambda_i$ and $\gamma_{j_i}$ respectively). 
	
	Since $\omega_u$ and $\omega^\sigma_u$ are surjective, $\widehat\omega_u$ is also surjective and so there is $\alpha_i \in \pi_1(S - A_u)$ such that $\lambda_i = \widehat\omega_u(\alpha_i)$. It follows that $\delta_i = \omega_u(\beta_i^{\epsilon_{j_i}})$ where $\beta_i = \alpha_i \xi_{j_i} \alpha_i^{-1}$ is an immersed meridian of $a_{j_i}$.
	
	So $\beta_i$ can be represented by the oriented boundary of an immersed disk $B_i \subset S$ which intersects $A_u$ only at $a_{j_i}$. The $B_i$'s can be chosen so that for a suitable embedded disk $B_0 \subset S - A_u$ which is a neighborhood of $*$, we have $B_0 \cap B_i \cong B^2$ and $B_0 \cap B_i \cap B_j = \{*\}$ for all $i \ne j \geq 1$.
	
	Now take $\eta'_i \in \pi_1(S - A_u)$ such that $\widehat\omega_u(\eta'_i) = \widehat\omega_f(\eta_i)$ for $i = 1,\dots, l$. We represent $\eta'_i$ by a transversely immersed loop (still denoted by $\eta'_i$) in $S - A_u$ such that $B_0 \cap \eta_i' \cong [0,1]$ and $B_0 \cap \eta_i' \cap \eta_j' = \{*\}$ for $i \ne j$.
	
	Consider the moves $t$, $t'$ and $t''$ of Figure~\ref{permutation/fig}, where $t$ acts on pairs $(B_i, B_j)$, $t'$ on pairs $(B_i, \eta'_j)$, and $t''$ on pairs $(\eta'_i, \eta'_j)$. These moves allow us to make the $B_i$'s and the loops $\eta'_i$'s intersect $B_0$ in the same order as the $D_i$'s and the $\eta_i$'s do with $D_0$.
	
	\begin{figure}
		\centering\includegraphics{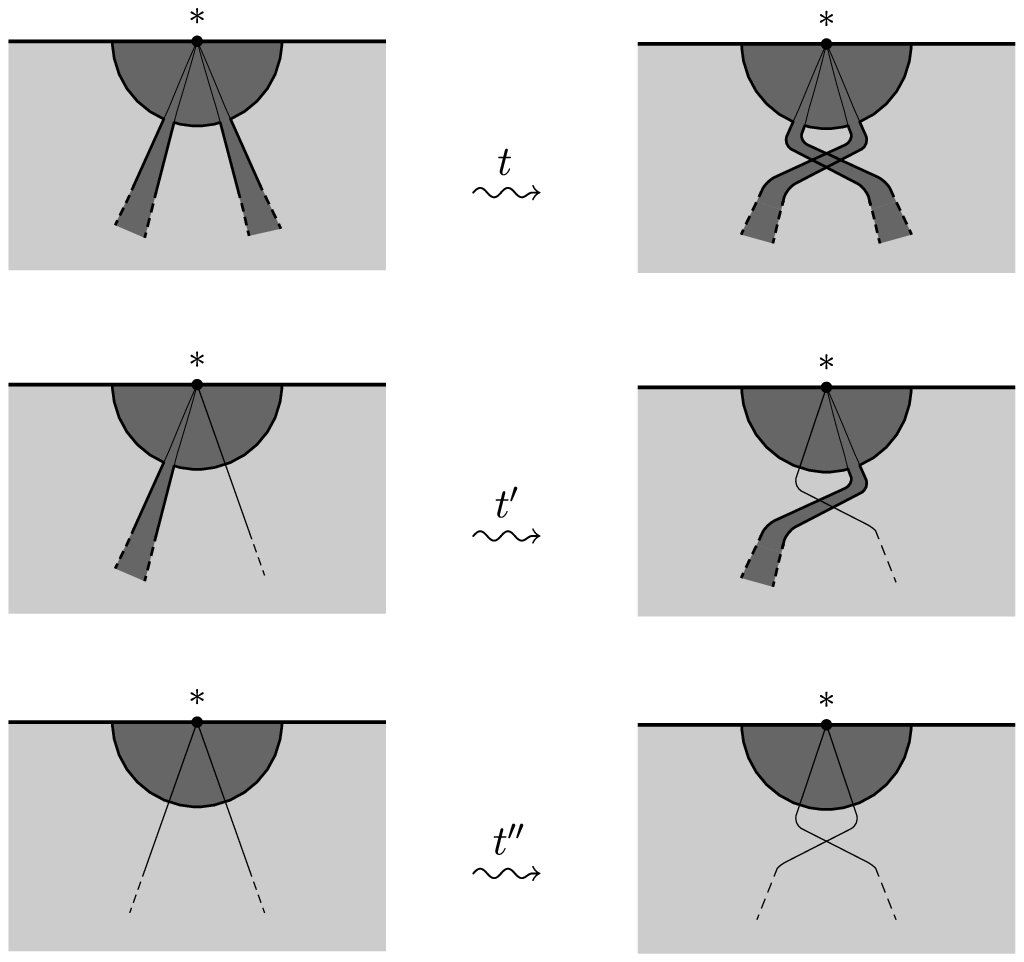}
		\caption{The permuting moves.}
		\label{permutation/fig}
	\end{figure}
	
	Take a map $\overline q : G' \to S$ which sends $D_0$ to $B_0$ diffeomorphically, immerses $D_i$ onto $B_i$ by preserving the orientation, and immerses the 1-handle $G'_i$ onto a regular neighborhood of $\eta'_i \subset S - A_u$ for all $i$. Assume also that $\overline q(A_f) = A_u$ and $\overline q(*') =*$. It follows that $\overline q$ is a $u$-regular immersion and that the homomorphism $\overline q_* : \pi_1(G' - A_f) \to \pi_1(S - A_u)$ induced by the restriction $\overline q_| : G' -A_f \to S - A_u$ satisfies $\overline q_*(\zeta_i) =\beta_i$ and $q_*(\eta_i) = \eta'_i$ for all $i \geq 1$.
	
	Now fix identifications $D'_i = \Cl(D_i - D_0) \cong [-1, 1] \times [-1, 1]$ for $i \geq 1$, such that the ordinate is 1 along the arc $D_0 \cap D_i'$ and such that the singular value $D_i \cap A_f$ has coordinates $\left(0, -{1 \over 2}\right)$.
	Let $r_i : G' \to G'$ be defined by the identity outside $D'_i$ and by the map of $D_i'$ to itself given by $r_i(t_1,t_2) = (t_1 t_2, t_2)$ up to the above identification. So $r_i$ shrinks a proper arc of $D_i'$ to a point, preserves the orientation above this arc and reverses the 
	orientation below it, as depicted in Figure~\ref{reverse/fig}. Moreover, $r_i(A_f) = A_f$ and the homomorphism induced by the restriction $r_{i*} : \pi_1(G' - A_f) \to \pi_1(G' - A_f)$ satisfies $r_{i*}(\zeta_i) = \zeta_i^{-1}$, $r_{i*}(\zeta_j) = \zeta_j$ for $j \neq i$, and $r_{i*}(\eta_j) = \eta_j$ for all $j$.

	\begin{figure}
		\centering\includegraphics{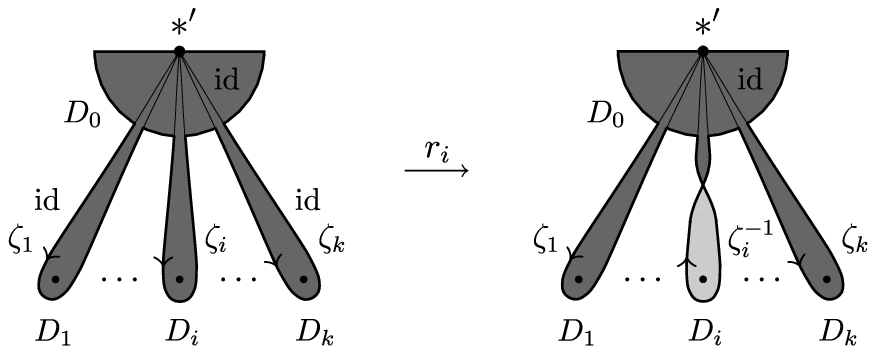}
		\caption{The twisting map $r_i$ on $D$.}
		\label{reverse/fig}
	\end{figure}
	
	Let $I = \{i_1, \dots, i_m\} \subset \{1, \dots, k\}$ be the set of those $i$ 
	such that $\epsilon_{j_i} \sigma_i  = -1$ and put $q = \overline q \circ r_{i_1} \circ \cdots \circ r_{i_m} : G \cong G' \to S$ (with $q = \overline q$ if $I$ is empty). Then $q$ is $u$-regular, $q^{-1}(A_u) = A_f$, $\omega_f = \omega_u \circ q_*$ and $\widehat\omega_f = \widehat\omega_u \circ q_*$ where $q_* : \pi_1(G - A_f) \to \pi_1(S - A_u)$ is induced by the restriction of $q$. It follows that $f \cong q^*(u)$.
	
	\medskip
	\noindent{\it `Only if' part.}
	Let $u : U \to S$ be universal with regular fiber $F_{g,b}$. Consider a Lefschetz fibration $f : V \to G$ with the same regular fiber and which satisfies the conditions $(1)$ and $(2)$. There is a $u$-regular 
	map $q : G \to S$ such that $q^*(u) = \widetilde f \cong f$. Then $\widetilde f$ satisfies the conditions $(1)$ and $(2)$ of the statement. Since $\omega_{\widetilde f} = \omega_u \circ q_*$ and $\omega^\sigma_{\widetilde f} = \omega^\sigma_u \circ q_*$ we obtain that $\omega_u$ and $\omega_u^\sigma$ are surjective.
		
	Consider now a class $\mathfrak c \in \cal C_{g,b}$. We can find a meridian $\zeta \in \pi_1(G - A_{\widetilde f})$ such that $\omega_{\widetilde f}(\zeta)$ is a Dehn twist around a curve $c \in \mathfrak c$. Since $q$ is $u$-regular, $q_*(\zeta) \in \pi_1(S - A_u)$ is an immersed meridian of $A_u$, hence $q_*(\zeta) = \alpha \xi \alpha^{-1}$ for a meridian $\xi$ of $A_u$ and for some $\alpha \in \pi_1(S - A_u)$. 

It follows that $\lambda = \widehat\omega_u(\alpha)$ satisfies $\lambda^{-1} \circ \omega_u(\xi) \circ \lambda = \omega_{\widetilde f}(\zeta)$ and so $\omega_u(\xi)$ is a Dehn twist around $\lambda(c) \in \mathfrak c$. Then $\mathfrak c$ can be represented by a vanishing cycle of $u$.
	
	\medskip
	
	The case of strongly Lefschetz fibrations can be handled similarly by tracing the same line of the previous proof. 
	We just give an idea of the `if' part:  if conditions $(1)$ and  $(2')$ are satisfied  then  the  $c_{j_i}$'s in the proof  of the first part can be chosen so that $\epsilon_{j_i} = \sigma_i$. Then the set $I$ defined above is empty and so $q$ is an orientation-preserving $u$-regular immersion such that $f = q^*(u)$.
	
	Finally, the last part of  the proposition  follows  since $\#\, \cal C_{g,b} = 1$ for $b \in \{0, 1\}$.
\end{proof}

\begin{proof}[Proof of Theorem~\ref{strong-univ/thm}]
	We consider three cases, depending on the values of $g$ and of $b$. 
	
	\noindent{\em Case I: $g \geq 2$ and $b = 1$.} Consider the Lefschetz 
	fibration $u_{g,1} : U_{g,1} \to B^2$ with regular fiber $F_{g,1}$ and with 
	monodromy sequence given by the $2g + 1$ signed vanishing cycles $(b_1^{-}$, $b_2$, $a_1,\dots, a_g, c_1^{-}, c_2, \dots, c_{g-1})$ depicted in Figure~\ref{lf/fig}, where a Hurwitz system is understood. In this figure the surface $F_{g,1}$ is embedded in $\R^3$ as part of the boundary of a standard genus $g$ handlebody.
	
	\begin{figure}
		\centering\includegraphics{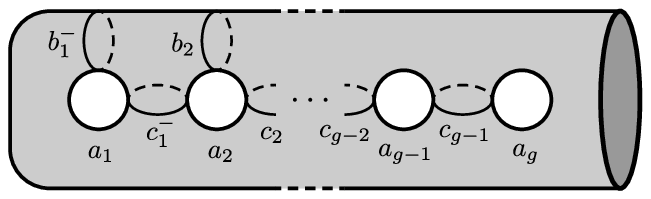}
		\caption{The Lefschetz fibration $u_{g,1}$ for $g \geq 2$.}
		\label{lf/fig}
	\end{figure}

	By a theorem of Wajnryb \cite{W99} $\M_{g,1}$ is generated by the $2g + 1$ Dehn twists $\alpha_i$, $\beta_i$ and $\gamma_i$ around the curves $a_i$, $b_i$ and $c_i$ respectively.
	It follows that $\omega_{u_{g,1}}$ is surjective and so $u_{g,1}$ is strongly universal by Proposition~\ref{main/thm}. 
	
	Now we analyze the 4-manifold $U_{g,1}$.
	In Figure~\ref{lefschetz2/fig} we pass from $u_{i,1}$ to $u_{i-1,1}$ in a two step destabilization process (first destabilize $a_i$ by the arc $s_i$ and then destabilize $c_{i-1}$ by the arc $s'_{i-1}$). This operation can be done whenever $i > 2$, and so by induction we can assume that $g=2$. In other words $U_{g,1} \cong U_{2,1}$ for $g > 2$.

	\begin{figure}
		\centering\includegraphics{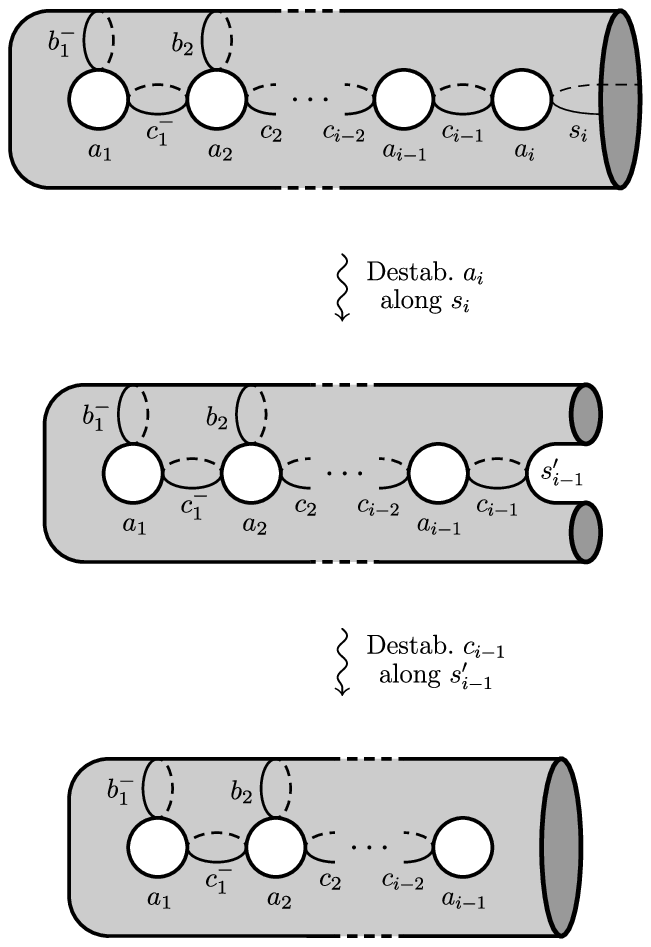}
		\caption{Simplifying $u_{i,1}$ for $i \geq 2$.}
		\label{lefschetz2/fig}
	\end{figure}
	
In Figure~\ref{lefschetz3/fig} we give some more destabilizations (first of $a_2$ along $s_2$ and then of $a_1$ along $s$), and finally we get the Lefschetz fibration depicted in the left lower part of the same figure. This has fiber $F_{0,3}$ and three boundary parallel vanishing cycles (two negative and one positive).

	\begin{figure}
		\centering\includegraphics{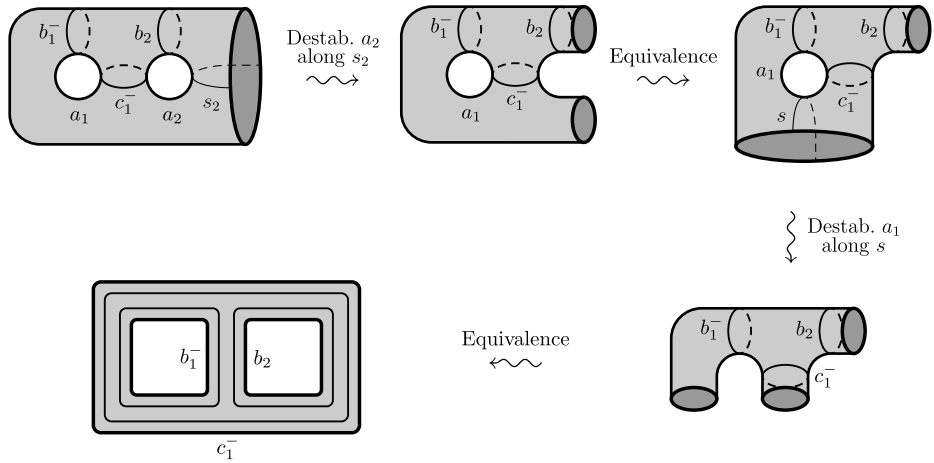}
		\caption{Simplifications of $u_{2,1}$.}
		\label{lefschetz3/fig}
	\end{figure}
	
	So a Kirby diagram for $U_{g,1}$ is that depicted in Figure~\ref{kirby/fig}, 
	which by a straightforward Kirby calculus argument can be recognized to be $M(O,1)$ (slide the outermost 2-handle over that with framing $-1$ so that the 
	latter cancels and by another simple sliding and deletion we get the picture 
	for $M(O,1)$ in the right side of Figure~\ref{kirby/fig}).
	
	\begin{figure}
		\centering\includegraphics{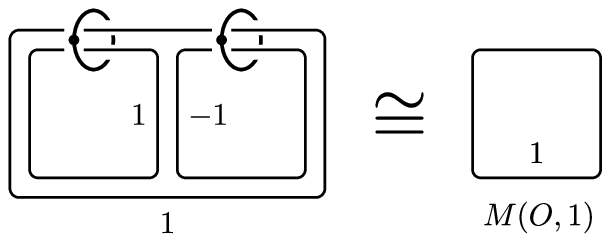}
		\caption{The manifold $U_{g, 1}\cong M(O,1)$ for $g \geq 2$.}
		\label{kirby/fig}
	\end{figure}

	\medskip
	
	\noindent{\em Case II: $g = b = 1$.} Consider the Lefschetz fibration 
	$u_{1,1} : U_{1,1} \to B^2$ with regular fiber $F_{1,1}$ and with monodromy 
	sequence $(a, b^-)$ depicted in Figure~\ref{torus/fig}.
	By \cite{W99} $\cal M_{1,1}$ is generated by the two Dehn twists $\alpha$ and $\beta$ around the curves $a$ and $b$ respectively and so Proposition~\ref{main/thm} implies that $u_{1,1}$ is strongly universal. 
	
	\begin{figure}
		\centering\includegraphics{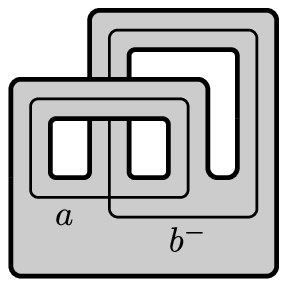}
		\caption{The Lefschetz fibration $u_{1,1}$.}
		\label{torus/fig}
	\end{figure}

	By a double destabilization we get a Lefschetz fibration with regular fiber $B^2$ and without singular values, hence $U_{1,1}$ is diffeomorphic to $B^2 \times B^2 \cong B^4$.
	
	\medskip
	
	\noindent{\em Case III: $(g,b) = (1,0)$.} Let $u_{1,0} : U_{1,0} \to B^2$ be 
	the Lefschetz fibration with fiber $F_{1,0} = T^2$ and with monodromy sequence $(a, b^-)$ depicted in Figure~\ref{torus-mon/fig}. Then $\omega_{u_{1,0}}$ is 
	surjective and so $u_{1,0}$ is strongly universal by Proposition~\ref{main/thm}.
	
	\begin{figure}
		\centering\includegraphics{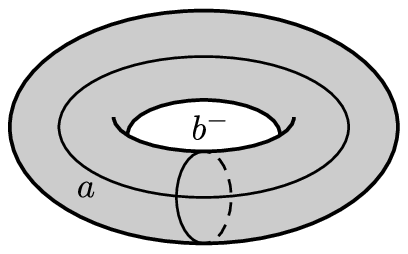}
		\caption{The Lefschetz fibration $u_{1,0}$.}
		\label{torus-mon/fig}
	\end{figure}
	
	Consider $F_{1,1} \subset \Bd U_{1,1} \cong S^3$ as the fiber of $u_{1,1}$ 
	over a point of $S^1 = \Bd B^2$. So $K = \Bd F_{1,1}$ is a knot in $S^3$. Moreover, by 
	pushing off $K$ along $F_{1,1}$ we get the framing zero (in terms of linking 
	number), since $F_{1,1}$ is a Seifert surface for $K$.
	Therefore the addition of a 2-handle to $B^4$ along $K$ with framing zero 
	produces $U_{1,0}$ \cite{GS99}. In Figure~\ref{kirby-M/fig} is depicted a Kirby diagram for $U_{1,0}$.
	This and the next three figures are referred to the 
	blackboard framing, namely that given by a push off along the 
	blackboard plus the extra full twists specified by the number near the knot. 
	Note that in Figure~\ref{kirby-M/fig} the blackboard framing coincides with 
	that of the fiber $F_{1,1}$.

	\begin{figure}
		\centering\includegraphics{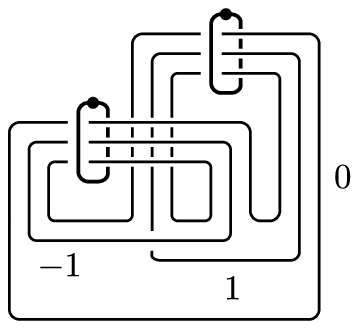}
		\caption{The manifold $U_{1,0}$.}
		\label{kirby-M/fig}
	\end{figure}
	
	Now we apply the moves $t_+$  and $t_-$ of Figures~\ref{trick/fig} and 
	\ref{trick2/fig} respectively, where the thick arcs with framing zero belong 
	to the same knot (which is assumed to be unlinked with the thin one) with the 
	orientations indicated in these figures. Such moves are proved by Kirby 
	calculus in the same figures.
	
	\begin{figure}
		\centering\includegraphics{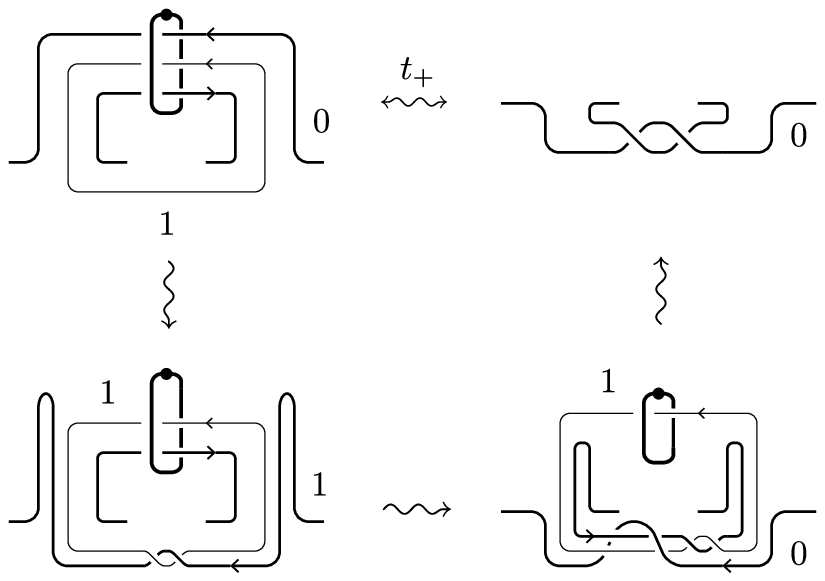}
		\caption{The move $t_+$.}
		\label{trick/fig}
	\end{figure}
	
	\begin{figure}
		\centering\includegraphics{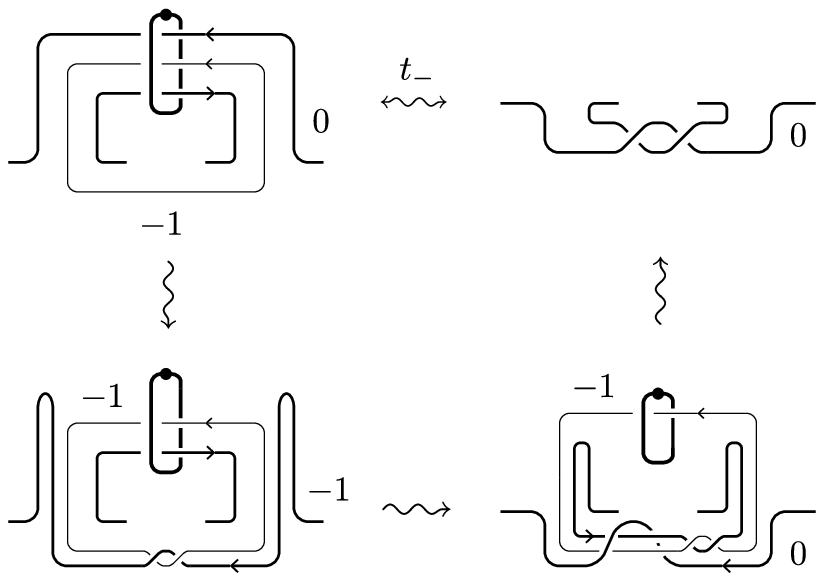}
		\caption{The move $t_-$.}
		\label{trick2/fig}
	\end{figure}
	
	We get Figure~\ref{torus-simpl/fig} where the two kinks in the second stage are 
	opposite and so do not affect the framing. The last stage, which gives the 
	figure eight knot, is obtained by framed isotopy.
	
	Since the writhe of the figure eight knot is zero, the blackboard framing zero is the same as linking number zero, and this concludes the proof.
\end{proof}

\begin{figure}
	\centering\includegraphics{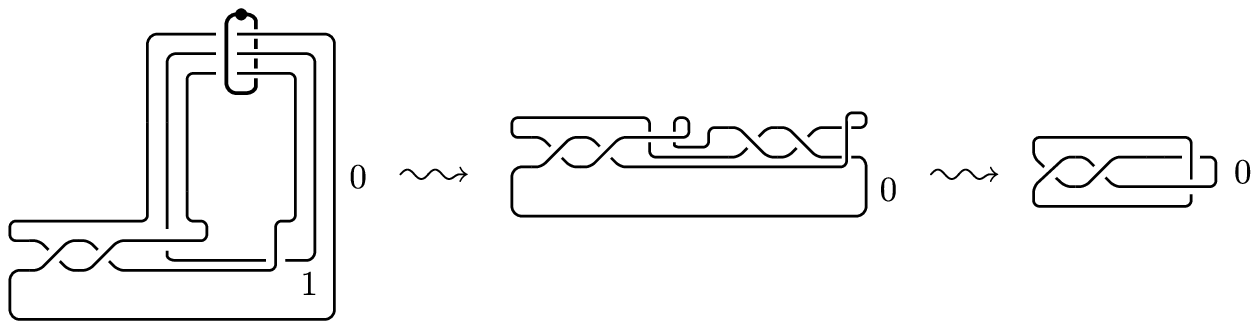}
	\caption{Kirby diagrams for $U_{1,0}$.}
	\label{torus-simpl/fig}
\end{figure}

\begin{proof}[Proof of Corollary~\ref{immers/cor}]
	By a theorem of Harer \cite{Harer1979} any 4-dimensional 2-handlebody $V$ admits a Lefschetz fibration $f : V \to B^2$ with regular fiber $F_{g,1}$ for some $g \geq 1$  (see also \cite{LP01} or \cite{EF06} for different proofs). Up to stabilizations we can assume $g \geq 2$.
	
	Theorem~\ref{strong-univ/thm} implies that $f \cong q^*(u_{g,1})$ for some orientation-preserving $u_{g,1}$-regular immersion $q : B^2 \to B^2$. Then we get a fibered immersion $\widetilde q : V \to U_{g,1}\cong M(O, 1)$. It is well-known that $M(O, 1)$ is orientation-preserving diffeomorphic to a tubular neighborhood of a projective line in $\CP^2$. Then we can consider $M(O, 1) \subset \CP^2$, and this concludes the proof.
\end{proof}

\begin{proof}[Proof of Corollary~\ref{paral/cor}]
	Let $F_{1,b}$ be the regular fiber of $f$. If $b = 1$ the corollary follows immediately from Theorem~\ref{strong-univ/thm} since $V$ admits a fibered immersion in $U_{1,1} \cong B^4 \subset \R^4$ and hence is parallelizable.
	
	If $b \geq 2$ we consider the 4-manifold $V'$ obtained from $V$ by the 
	addition of 2-handles along all but one boundary components of the regular fiber $F_{1,b} \subset \Bd V$ with framing zero with respect to $F_{1,b}$. Then $V \subset V'$. Moreover, $f$ extends to a Lefschetz fibration $f' : V' \to B^2$ with regular fiber $F_{1,1}$ whose monodromy is obtained from $\omega_f$ by composition with the homomorphism from $\M_{1,b}$ to $\M_{1,1}$ induced by capping off by disks all but one boundaries components of $F_{1,b}$. 
	
	The non-separating assumption on the vanishing cycles of $f$ implies that $f'$ is allowable, and so $V'$ immerses in $\R^4$ by the first case.
	
	If $b = 0$, $V$ fibered immerses in the manifold $U_{1,0} \cong M(E, 0)$ of Theorem~\ref{strong-univ/thm}. We conclude by observing that $M(E,0)$ immerses in $\R^4$ as a tubular neighborhood of $B^4 \cup D$ where $D \subset \R^4 - \Int B^4$ is a self-transverse immersed disk with boundary the knot $E$.
\end{proof}

\section{Final remarks}\label{remarks/sec}

	\begin{remark}
	
	It is not difficult to see  that for $b \geq 1$
	$$\#\, \cal C_{g,b} = 
	\begin{cases}
		\displaystyle{\left\lfloor {b \over 2} \right\rfloor}, & \text{if $g = 0$} \cr
		\cr
		\displaystyle\left\lfloor {g b - g + b \over 2} \right\rfloor + 1, & \text{if $g \geq 1$}
	\end{cases}$$
	which is a lower bound for the number of singular points of a universal Lefschetz fibration with fiber $F_{g,b}$ $(\lfloor x \rfloor$ denotes the integer part of $x \in \R)$.
\end{remark}

\begin{remark}
	In order to include also the not allowable Lefschetz fibrations it suffices to replace, in Proposition~\ref{main/thm}, $\cal C_{g,b}$ with the set  of $\widehat\M_{g,b}$-equivalence classes of homotopically essential curves. The proof is very similar.
\end{remark}

In \cite{LP01} Loi and Piergallini characterized compact Stein domains of 
dimension four, up to orientation-preserving diffeomorphisms, as the total 
spaces of positive Lefschetz fibrations (meaning with only positive singular points) over $B^2$ with bounded fiber. We can express this theorem in terms of universal positive Lefschetz fibrations.

Following the notations of the proof of Theorem~\ref{strong-univ/thm} let 
$p_g : P_g \to B^2$ be the Lefschetz fibration with fiber $F_{g,1}$ and 
monodromy sequence given by $(a, b)$ for $g = 1$ and $(b_1$, $b_2$, 
$a_1,\dots, a_g, c_1, c_2, \dots, c_{g-1})$ for $g \geq 2$ as showed in Figure~\ref{univ-stein/fig}.

\begin{figure}
	\centering\includegraphics{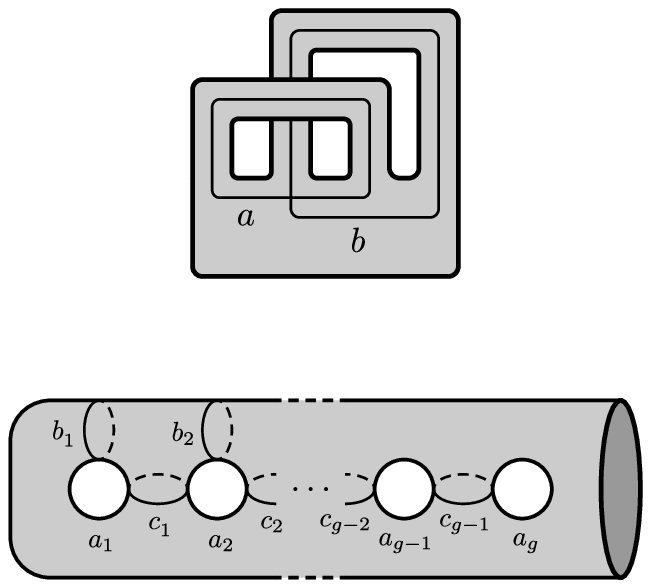}
	\caption{The positive universal Lefschetz fibrations $p_g$.}
	\label{univ-stein/fig}
\end{figure}

Then $p_g$ is universal (but not strongly universal) by Proposition~\ref{main/thm}. Moreover $P_1 \cong B^4$ and $P_g$ has the Kirby diagram depicted in Figure~\ref{kirby-stein/fig} for $g \geq 2$. That $P_g \cong M(O, -3)$ follows by the same argument used in the proof of Theorem~\ref{strong-univ/thm}.

\begin{figure}
	\centering\includegraphics{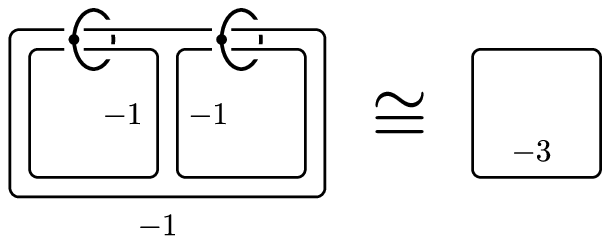}
	\caption{The 4-manifold $P_g$ for $g \geq 2$.}
	\label{kirby-stein/fig}
\end{figure}

However, $SL(p_g)$ is the set of equivalence classes of all positive Lefschetz 
fibrations with fiber $F_{g,1}$. The proof is exactly the same of Proposition~\ref{main/thm} with $\epsilon_{j_i} = \sigma_i =1$.

Of course any Lefschetz fibration with bounded fiber can be positively 
stabilized so that the fiber has connected boundary.

It follows that compact 4-dimensional Stein domains with strictly pseudoconvex 
boundary coincide, up to orientation-preserving diffeomorphisms, with the 
total spaces of Lefschetz fibrations that belong in $SL(p_g)$ for some $g \geq 
1$.

	%%%%%%%%%%%%%%%%%%%%%%%%%%%%%%%%


\begin{thebibliography}{00}
	
	\bibitem{APZ2011}
		N. Apostolakis, R. Piergallini and D. Zuddas, {\em Lefschetz fibrations over the disk}, preprint, arXiv (2011).
	
	\bibitem{EF06}
		J.B. Etnyre and T. Fuller, {\em Realizing 4-manifolds as achiral Lefschetz fibrations}, {Int. Math. Res. Not.} (2006), 1--21.
	
	\bibitem{GS99}
		R.E. Gompf and A.I. Stipsicz, {\em {$4$}-manifolds and {K}irby calculus}, volume~20 of {Graduate Studies in Mathematics}, American Mathematical Society, 1999.
	
	\bibitem{Gr73}
		A. Gramain, {\em Le type d'homotopie du groupe des diff\'eomorphismes d'une surface compacte}, Ann. Sci. \'Ec. Norm. Sup\'er. {\bf 6} (1973), 53--66.
	
	\bibitem{Harer1979}
		J. Harer, {\em Pencils of curves on $4$-manifolds}, PhD thesis, University of California, Berkeley, 1979.
	
	\bibitem{L97}
		W.B.R. Lickorish, {\em An introduction to knot theory}, Graduate Texts in Mathematics {\bf 175} (1997), Springer-Verlag.
	
	\bibitem{LP01}
		A. Loi and R. Piergallini, {\em Compact Stein surfaces with boundary as branched covers of $B^4$}, Invent. Math. {\bf 143} (2001), 325--348.
	
	
	\bibitem{MS74}
		J.W. Milnor and J.D. Stasheff, {\em Characteristic classes}, Annals of Mathematics Studies {\bf 76} (1974), Princeton University Press.
	
	
	\bibitem{P67}
		A. Phillips, {\em Submersions of open manifolds}, Topology {\bf 6} (1967), 171--206.
	
	
	
	\bibitem{W99}
		B. Wajnryb, {\em An elementary approach to the mapping class group of a surface} Geom. Topol. {\bf 3} (1999), 405--466.
	
	\bibitem{Z09}
		D. Zuddas, {\em Representing {D}ehn twists with branched coverings}, {Int. Math. Res. Not.} 3 (2009), 387--413.
	

\end{thebibliography}
\end{document}